\documentclass{amsart}

\usepackage{amscd}
\usepackage{amsfonts}
\usepackage{amsmath}
\usepackage{amssymb}
\usepackage{amsthm}
\usepackage{amsxtra}
\usepackage{array}
\usepackage{fancybox}
\usepackage{fancyhdr}
\usepackage{graphics}
\usepackage{latexsym}
\usepackage{mathrsfs}
\usepackage{upref}
\usepackage{verbatim}
\usepackage[arrow, matrix, curve]{xy}

\newtheorem{lemma}{Lemma}[section]
\newtheorem{theorem}[lemma]{Theorem}

\newtheorem{definition}[lemma]{Definition}
\newtheorem{example}[lemma]{Example}

\newcommand{\M}{\mathscr{M}}
\newcommand{\m}{\mathfrak{m}}
\newcommand{\p}{\mathfrak{p}}

\makeindex

\begin{document}
\title{Graded $F$-modules and Local Cohomology}
\author{Yi Zhang}
\address{Dept. of Mathematics, University of Minnesota, Minneapolis
MN 55455}
\email{zhang397@umn.edu}
\thanks{The author thanks his advisor Professor Gennady
Lyubeznik for his guidance, support and time. NSF support through grant
DMS-0701127
is gratefully acknowledged.}
\date{}

\begin{abstract}
Let $R=k[x_1,\cdots, x_n]$ be a polynomial ring over a field $k$ of
characteristic $p>0,$ let $\m=(x_1,\cdots, x_n)$ be the maximal ideal generated
by the variables, let $^*E$
be the naturally graded injective hull of $R/\m$ and let $^*E(n)$ be $^*E$
degree
shifted downward by $n.$ 
We introduce the notion of graded $F$-modules (as a refinement of the notion of
$F$-modules) and show that if a graded $F$-module $\M$ has zero-dimensional
support, then $\M,$ as a graded $R$-module, is isomorphic to a direct sum of a
(possibly infinite)
number of copies of $^*E(n).$ 

As a consequence, we show that if the functors $T_1,\cdots,T_s$ and $T$ are
defined by
$T_{j}=H^{i_j}_{I_j}(-)$ and $T=T_1\circ\cdots \circ T_s,$ where $I_1,\cdots,
I_s$ are homogeneous ideals of $R,$
then
as a naturally graded $R$-module, the local cohomology module
$H^{i_0}_{\m}(T(R))$ is isomorphic to
$^*E(n)^c,$ where $c$ is a finite number. If $\text{char}k=0,$ this
question is open even for $s=1.$
\end{abstract}
\maketitle

\section{Introduction}

Throughout this paper, $R=k[x_1,\cdots, x_n]$ is a polynomial ring over a field
$k,$ $\m=(x_1,\cdots, x_n)$ is the maximal homogeneous ideal, and $^*E$ is the
$^*$injective hull of $R/\m.$ 

It is well-known that
$H^i_\m(H^j_I(R)),$ for an ideal $I$ of $R,$ and more generally,
$H^{i_0}_{\m}(T(R))$ where $T$ is the composition of functors $T_1,\cdots, T_s$
with $T_{j}=H^{i_j}_{I_j}(-),$ are isomorphic
to a direct sum of a
finite number of copies of
$E$ (the ungraded injective hull of $R/\m$).
In the case $\text{char}k>0,$ this is due to Huneke and Sharp \cite{HS93}, in
the case $\text{char}k=0,$ this is due to Lyubeznik \cite{gL93}.

If $I_1,\cdots,I_s \subset R$ are homogeneous, the local cohomology module
$H^{i_0}_\m(T(R))$ acquires a natural grading. This paper is motivated by the
following question: {\it How is
this grading related to the natural grading on $^*E?$} We show in Theorem
\ref{TG}
that if $\text{char}k>0,$ then
$H^{i_0}_{\m}(T(R))$ is
isomorphic, as a graded $R$-module,
to a direct sum of a finite number of copies of $^*E(n),$ that is $^*E$
degree shifted downward by $n.$ If $\text{char}k=0,$ the question is open even
for $s=1.$

Our proof is based on a new notion of graded $F$-modules which is a fairly
straightforward graded version of $F$-modules introduced in \cite{gL97}. The
local cohomology modules $H^j_I(R)$ and $H^{i_0}_\m(T(R))$ carry a natural
structure of graded $F$-modules. Our
main result (Theorem \ref{T:GFS}) says that a graded $F$-module supported in
dimension 0 is
isomorphic, as a graded $R$-module, to a direct sum of copies of $^*E(n).$
The above-mentioned Theorem \ref{TG} about $H^{i_0}_\m(T(R))$ is a
straightforward
consequence of
this result.

\section{Preliminaries}

For the rest of this paper, we assume that $k$ is of characteristic
$p>0.$ The Frobenius homomorphism
$R\xrightarrow{r\mapsto r^{p}}R',$ where $R'$ is another copy of $R,$ induces
the Frobenius functor $F:R\text{-mod}
\to R\text{-mod}$ as the pull back functor, that is $F(M)=R'\otimes_R M$ and
$F(M\xrightarrow{f} N)=(R'\otimes_R M\xrightarrow{\text{id} \otimes_R f}
R'\otimes_R N).$ We follow \cite{gL97} for the
$F$-module
theory.

\begin{definition} \cite[Definition 1.1]{gL97}
An $F$-module is an $R$-module $\mathscr{M}$ equipped with an
$R$-module isomorphism $\theta:\mathscr{M}\rightarrow
F(\mathscr{M})$ called the structure morphism of
$\mathscr{M}.$ 

A homomorphism of $F$-modules is an $R$-module
homomorphism $f:\mathscr{M}\rightarrow\mathscr{M}'$ such that the
following diagram commutes:
$$\xymatrix{
  \mathscr{M}    \ar[d]_{\theta} \ar[r]^{f} & \mathscr{M}' \ar[d]^{\theta'} \\
  F(\mathscr{M}) \ar[r]_{F(f)}              & F(\mathscr{M}'),
  }$$
where $\theta$ and $\theta'$ are the structure morphisms of
$\mathscr{M}$ and $\mathscr{M}'.$
\end{definition}

Observe that the ring $R=k[x_1,\cdots, x_n]$ has a natural grading
$R=\bigoplus_{i\in \mathbb{N}} R_i$
(as a $\mathbb{Z}$-module) such that $R_i$ consists of all homogeneous
polynomials in
$x_1,\cdots, x_n$ of degree $i.$ Recall that a graded $R$-module is an
$R$-module $M$ together with a
decomposition
$M=\bigoplus_{i\in \mathbb{Z}} M_i$ (as a $\mathbb{Z}$-module) such that
$R_iM_j= M_{i+j}$ for
all $i,j\in\mathbb{Z}.$ Recall if $M$ and $N$ are both graded $R$-modules, then
a homomorphism $\varphi: M\to N$ is degree preserving if $\varphi(M_i)\subseteq
N_i$ for all $i\in \mathbb{Z}.$ 

If $\M$ is graded, we define the
grading of $F(\M)$ by $\deg r\otimes x=\deg r+p\cdot \deg x$ for all
homogeneous $r\in R$ and $x\in \M.$ Now we introduce a definition of 
graded $F$-modules as follows:

\begin{definition}
An $F$-module $(\M,\theta)$ is graded if $\M$ is a graded $R$-module
and the structure isomorphism $\theta:\mathscr{M}\rightarrow
F(\mathscr{M})$ is degree preserving. A homomorphism of
graded $F$-modules $f:\M\to \M'$ is a degree preserving $F$-module homomorphism.
\end{definition}

\begin{example} \label{E:GFR}
The canonical $F$-module structure on $R$ defined by the $R$-module 
isomorphism $\theta:R\xrightarrow{r\to r\otimes1}F(R)$ \cite[Page 72]{gL97}
makes $(R,\theta)$ a
graded $F$-module.
\end{example}

The theory of $F$-modules developed in \cite{gL97} can be developed in this
graded
version without difficulty. In particular, it is easily seen that the
category of graded $F$-modules is abelian. The facts we need in this paper are
the
following with the use of the standard terminology in \cite[Section
3.6]{BH93}:


\begin{theorem} \label{T:GFLC}
If $\M$ is a graded
$F$-module, then there is an induced graded $F$-module structure on
the local cohomology modules $H^i_I(\M)$ for any homogeneous ideal $I$ of $R.$
\end{theorem}    
\begin{proof}
Since the ordinary local cohomology can be computed using $^*$injective
resolutions \cite[Corollary 12.3.3]{BS98}, the proof is basically the same as in
\cite[Example 1.2(b)]{gL97} except that
instead of injective resolutions one uses $^*$injective ones.
\end{proof}

\begin{theorem} \label{T:Sup0}
If $\M$ is a graded $F$-module such that $\dim_R\text{Supp} \M=0,$ then $\M$ is
a $^*$injective $R$-module.
\end{theorem}
\begin{proof}
A proof of this is, with minor and straightforward modifications, the same as
the proof of the $\dim_R\text{Supp} \M=0$ case of \cite[Theorem 1.4]{gL97}.
Modifications involve choosing the elemnents $e_i$ and $e_{i,j}$ homogeneous and
in the last step showing that $M_i$ is isomorphic to $^*E(R/\m)(t)$ where
$t=\deg
e_i$ (rather than just $E(R/\m),$ as in \cite[Theorem 1.4]{gL97}).
\end{proof}

\section{The Main Result}

From this section on, we will adopt the notation $F^*$ to represent the
Frobenius functor $F.$ Let us first recall a result about the adjointness
between $F_*^l$ and
$F^{*^l}.$ We denote the source
and target of $F^l$ by $R_s$ and $R_t$ 
respectively, that is $F^l:R_s\rightarrow R_t.$ There are two
associated functors
$$F^{*^l}:R_s\text{-mod}\rightarrow R_t\text{-mod}$$ such that
$F^{*^l}(-)=R_t\otimes_{R_s}-,$ and
$$F_*^l:R_t\text{-mod}\rightarrow R_s\text{-mod}$$ which is the
restriction of scalars. 

Denote the multi-index
$(i_1,\cdots,i_n)$
by $\bar{i},$
especially $\overline{p^l-1}=(p^l-1,\cdots,p^l-1).$ When $k$ is perfect, $R_t$
is a free $R_s$-module
on the $p^{ln}$ monomials $e_{\bar{i}}=x^{i_1}_1\cdots x^{i_n}_n$
where \hbox{$0\leqslant i_j<p^l$} for every $j.$ Suppose $M$ is an $R_t$-module
and $N$
is an $R_s$-module. For each \hbox{$f\in \text{Hom}_{R_t}(M,F^{*^l}(N)),$}
define
$f_{\bar{i}}=p_{\bar{i}}\circ f:F_*^l(M)\rightarrow N,$ where
$$F^{*^l}(N)(=\bigoplus_{\bar{i}}(e_{\bar{i}}\otimes_{R_s}N))\xrightarrow{
y\mapsto
e_{\bar{i}}\otimes p_{\bar{i}}(y)}
e_{\bar{i}}\otimes_{R_s}N$$ is the natural projection to the
$\bar{i}$-component. 
The duality theorem in~\cite{gL09} says:

\begin{theorem}\label{TD} (Theorem 3.3 in~\cite{gL09})
When $k$ is perfect, for every $R_t$-module $M$ and every $R_s$-module $N,$
there is an
$R_t$-linear isomorphism
\begin{align*}
\text{Hom}_{R_s}(F^l_*(M),N) &\cong \text{Hom}_{R_t}(M,
F^{*^l}(N))\\
g_{\overline{p^l-1}}(-)&\leftarrow
(g=\oplus_{\bar{i}}(e_{\bar{i}}\otimes_{R_s}g_{\bar{i}}(-)))\\
g&\mapsto\oplus_{\bar{i}}(e_{\bar{i}}\otimes_{R_s}g(e_{\overline{p^l-1}-\bar{i}}
(-))).
\end{align*}
\end{theorem}

We use this duality theorem to prove the following striking result. 

\begin{theorem} \label{T:VD}
Let $M$ be a graded $R$-module. Assume
\begin{enumerate}
\item $\{\:d\in \mathbb{Z}\:|\:M_d\neq0\:\}$ is finite;
\item $M_{-n}=0.$
\end{enumerate}
Then there is $s\in \mathbb{N}$ (that depends only on the set $\{\:d\in
\mathbb{Z}\:|\:M_d\neq0\:\}$) such that for any $l\geqslant s$ and for any
graded
$R$-module $N,$ the only degree preserving $R$-module map $f:M\to F^{*^l}(N)$ is
the zero map.
\end{theorem}
\begin{proof}
Let $K$ be the perfect closure of $k.$ Viewing $K\otimes_k R,$ $K\otimes_kM$
and $K\otimes_kN$ as the new $R,$ $M$ and $N,$ we may assume that $k$ is
perfect. Therefore Theorem \ref{TD} applies and it is sufficient to
study the $\overline{p^l-1}$ component
of the image of $M.$ Recall that the Frobenius functor $F^{*^l}$ multiplies the
grading
by $p^l,$ i.e. $$\deg r\otimes x= \deg r + p^l\cdot \deg x, \quad (r\in R_t,\
x\in N\ \text{and}\ r\otimes x\in F^{*^l}(N)).$$ Since
$F^{*^l}(N)=\bigoplus_{\bar{i}} (e_{\bar{i}} \otimes_{R_s} N)$ and $\deg
e_{\bar{i}}=\sum_j i_j,$
for every $d\in \mathbb{Z}$ we have $$F^{*^l}(N)_d=\bigoplus_{\bar{i}}
(e_{\bar{i}} \otimes_{R_s} N)_d= \bigoplus_{\bar{i}}
e_{\bar{i}} \otimes_{R_s} N_{(d-\deg e_{\bar{i}})/p^l},$$ where the direct sum
is taken over those $\bar{i}$ for which $(d-\deg
e_{\bar{i}})/p^l$ is an integer. Clearly, $\deg e_{\overline{p^l-1}}=n(p^l-1).$
When $d\neq -n$ and $l$ is
sufficiently large, the fraction
$(d-n(p^l-1))/p^l$ is not an integer. Hence the coefficient of
$e_{\overline{p^l-1}}$
in $F^{*^l}(N)_d$ is 0. Let $d$ run through the
finite set $\{\:d\in
\mathbb{Z}\:|\:M_d\neq0\:\}$ and enlarge $l$ correspondingly, we see that
the $\overline{p^l-1}$ component of the image of $M$ is 0. Moreover, it is
obvious that the
selection of $l$ depends only on the set $\{\:d\in
\mathbb{Z}\:|\:M_d\neq0\:\}.$ Now Theorem \ref{TD} induces the
conclusion.
\end{proof}

\begin{theorem} \label{T:GFS}
Let $\mathscr{M}$ be a graded $F$-module supported on
$\m=(x_1,\cdots,x_n).$
Then $\mathscr{M}$ as a graded
$R$-module is a direct sum of a
(possibly infinite) number of copies of $^*E(n).$
\end{theorem}
\begin{proof}
Since $\M$ is supported on $\m,$ it is $^*$injective by Theorem
\ref{T:Sup0}. By
\cite[Theorem 3.6.3]{BH93}, every $^*$injective module can be decomposed
into a direct sum of modules $^*E(R/\p)(i)$ for
graded prime ideals $\p\in \text{Spec}R$ and integers $i\in \mathbb{Z}.$ Since
$\M$ is supported on $\m,$ the only $\p$ that appears in the decomposition is
$\p=\m,$ i.e. $\M= \oplus_i ^*E(i)^{\alpha(i)}$ where $^*E=
{}^*E(R/\m)$ and
$\alpha(i)$ is the
(possibly infinite) number of copies of $^*E(i).$
Let $\theta:
\M\to F^*(\M)$ be the structure isomorphism of $\M.$ Fix $i\neq n,$ assume
$\alpha(i)\neq 0,$ i.e. $\text{soc}^*E(i)^{\alpha(i)}\neq0,$ and apply Theorem
\ref{T:VD}
to $M=\text{soc}
^*E(i)^{\alpha(i)}$ and $N=\M.$ Since the degree of the
socle of $^*E(i)$ is $-i\neq -n,$ we see that
the composition of isomorphisms $F^{*^{l}}(\theta)\circ
F^{*^{l-1}}(\theta)\circ \cdots\circ \theta: \M\to F^{*^l}(\M)$
vanishes on $M,$ i.e. $\theta$ is not an isomorphism. That is a contradiction.
Hence $\alpha(i)=0$ when $i\neq n.$
\end{proof}

\begin{theorem} \label{TG}
Let the functors $T_1,\cdots,T_s$ and $T$ be defined by
$T_{j}=H^{i_j}_{I_j}(-)$ and $T=T_1\circ\cdots \circ T_s,$ where $I_1,\cdots,
I_s$ are homogeneous ideals of $R.$ Then as a graded $R$-module,
$H^{i_0}_\m(T(R))$ is isomorphic to $^*E(n)^c$ for some $c<\infty.$
\end{theorem}
\begin{proof} The graded $F$-module structure on $R$ in Example \ref{E:GFR}
induces a graded
$F$-module structure on $H^{i_0}_\m(T(R))$ by induction on $s$ via Theorem
\ref{T:GFLC}. Now Theorem \ref{T:GFS} gives
the
desired result.
\end{proof}


\begin{thebibliography}{99}

\bibitem{BS98}
M. P. Brodmann, R. Y. Sharp, Local cohomology: an algebraic introduction with
geometric applications. Cambridge Studies in Advanced Mathematics, 60. Cambridge
University Press, Cambridge, 1998.

\bibitem{BH93}
W. Bruns, J. Herzog, Cohen-Macaulay rings. Cambridge Studies in
Advanced Mathematics, 39. Cambridge University Press, Cambridge, 1993.

\bibitem{HS93}
C. Huneke, R. Sharp, Bass numbers of local cohomology modules.  Trans. Amer.
Math. Soc.  339  (1993),  no. 2, 765--779.

\bibitem{gL93}
G. Lyubeznik, Finiteness properties of local cohomology modules (an application
of $D$-modules to commutative algebra).  Invent. Math.  113  (1993),  no. 1,
41--55.

\bibitem{gL97}
G. Lyubeznik, $F$-modules: applications to local cohomology and
$D$-modules in characteristic $p>0$.  J. Reine Angew. Math.  491
(1997), 65--130.

\bibitem{gL09}
G. Lyubeznik, W. Zhang, Y. Zhang, A property of the Frobenius map of
a polynomial ring. Preprint, arXiv: 1001.2949.


\end{thebibliography}
\end{document}